\documentclass[fleqn,preprint,3p,a4paper]{elsarticle}
\UseRawInputEncoding

\usepackage{amssymb}
\usepackage{amsmath}
\usepackage{amsthm}
\usepackage{dcolumn}
\usepackage{endnotes}
\usepackage{tabularx}
\usepackage[matrix,arrow]{xy}
\usepackage{wasysym}
\usepackage[colorlinks,
            linkcolor=red,
            anchorcolor=red,
            citecolor=red]{hyperref}

\theoremstyle{plain}

  \newtheorem{thm}{Theorem}[section]
  \newtheorem{lem}[thm]{Lemma}
  \newtheorem{prop}[thm]{Proposition}
  
\theoremstyle{definition}
  \newtheorem{defn}[thm]{Definition}

  \newtheorem{rem}[thm]{Remark}

\theoremstyle{break}

\newcommand{\Ua}[0]{\makebox[1ex][l]{\lower.15ex
                                 \hbox{$\uparrow$}}\kern-1ex\lower-.15ex
                                 \hbox{$\uparrow$}}

\newcommand{\Da}[0]{\makebox[1ex][l]{\lower.15ex
                                 \hbox{$\downarrow$}}\kern-1ex\lower-.15ex
                                 \hbox{$\downarrow$}}
\journal{}

\begin{document}

\begin{frontmatter}



\title{QFS-space and its properties\tnoteref{t1}}
\tnotetext[t1]{This research was supported by the National Natural Science Foundation of China (Nos. 12071199, 12071188, 11661057)}.

\author[Wang]{Wu Wang\corref{mycorrespondingauthor}}
\cortext[mycorrespondingauthor]{Corresponding author}
\ead{wangwu@alu.scu.edu.cn}
\address[Wang]{Basic Course Department, Zhonghuan Information College Tianjin University of Technology, Tinjin 300380, China}

\begin{abstract}
In this paper, the concept of quasi-finitely separating map and quasiapproximate identity are introduced. Based on these concepts, QFS-spaces and quasicontinuous maps are defined. Properties and characterizations of QFS-spaces are explored. Main results are: (1) Each QFS-space is  quasicontinuous space; (2) Closed subspaces, quasicontinuous projection spaces of QFS-spaces are QFS-spaces; (3) Continuous retracts of QFS-spaces are QFS-spaces and a kind of retracts of QFS-spaces are constructed; (4) Upper powerspaces of continuous QFS-spaces are FS-spaces.
\end{abstract}





\end{frontmatter}


\section{Introduction}

Domain theory was first introduced by Dana Scott in the early 1970s, and the main purpose is to provide a mathematical tool for the semantics of functional programming languages, see\cite{1,2,3,20}. The most distinctive feature of domain theory is that it integrates order structures, topology structures and
computer science\cite{4,6,7,15,16,17}. The main objects of domain theory are posets and domains, see\cite{4,5,8}. In \cite{16}, Directed space is defined by Hui Kou. It is easy to see that directed spaces are equivalent to monotone determined spaces, which is defined in \cite{17}. $DTOP$, the category of directed spaces was shown to be cartesian closed.

In \cite{9}, Chen, Kou and Lyu investigated two approximation relations on a $T_{0}$ topological space, the n-approximation and the d-pproximation.
Different kinds of continuous spaces are defined by the two approximations. They showed that a $T_{0}$ topological space is n-continuous iff it is a c-space iff it  is d-continuous directed space. In order to find the maximal full cartesian closed subcategories of d-continuous directed space, Xie and Kou \cite{13} introduced the concept of FS-space.

As a generalization of d-continuous spaces, the concept of quasicontinuous space was introduced in \cite{21} by Feng and Kou. There are many examples of directed spaces which are quasicontinuous spaces but not continuous spaces. Therefore, a quasicontinuous space may not be a continuous space actually, let alone an FS-space. Now questions naturally arise: what is the counterpart of an FS-space in the realm of quasicontinuous spaces and to what extent does the counterpart share nice properties possessed by FS-spaces? In this paper we manage to give answers to these questions. First, we introduce the concepts of quasi-finitely separating map and quasi-approximate identity. Then QFS-spaces and quasicontinuous maps are defined accordingly, and properties as well as characterizations of QFS-domains are explored. Main results are:  (1) Each QFS-space is  quasicontinuous space; (2) Closed subspaces, quasicontinuous projection spaces of QFS-spaces are QFS-spaces; (3) Continuous retracts of QFS-spaces are QFS-spaces and a kind of retracts of QFS-spaces are constructed; (4) Upper powerspaces of continuous QFS-spaces are FS-spaces.

\section{Preliminaries}

Now, we introduce the concepts needed in this paper. The readers can also consult\cite{8,11,12,14}. A nonempty set $L$ endowed with a partial order $\leq$ is called  poset.
A subset $D\subseteq L$ is called directed set if for any $x,y\in D$, there exists $d\in D$ such that $x, y\leq d$. A poset $L$ is called a directed complete poset ($dcpo$, for short) if any directed subset of $L$ has a sup in $L$. For any $x, y\in L$, we say that $x$ is way below $y$ (denoted by $x\ll_{s} y$) if for any directed set $D$ of $L$, $y\leq \vee D$ implies that there is some $d\in D$ with $x\leq d$. A poset $L$ is called continuous if for any $x\in L$, $\{a\in L: a\ll_{s} x\}$ is directed set and has $x$ as its supremum. For a subset $A$ of $L$, let $\uparrow A =\{x\in L: \exists a\in A, a\leq x\}$, $\downarrow A =\{x\in L:\exists a\in A, x\leq a\}$. We use $\uparrow a$ (resp.$\downarrow a$) instead of $\uparrow \{a\}$(resp. $\downarrow \{a\}$) when $A=\{a\}$. $A$ is called an upper (resp. a lower) set if $A=\uparrow A$ (resp. $A=\downarrow A$).

Let $L$ be a poset and $U\subseteq L$. Then $U$ is called Scott open iff it satisfies: (1) $U=\uparrow U$; (2) For any directed sets $D\subseteq L$, $\vee D\in U$ implies $D\cap U\neq\emptyset$. The collection of all Scott open subsets of $L$ is called the Scott topology of $L$ and denoted by $\sigma(L)$.

In this paper, topological spaces will always be supposed to be $T_{0}$ spaces. For a topological space $X$, its topology is denoted
by $\tau$. The partial order $\leq$ defined on $X$ by $x\leq y \Leftrightarrow x\in cl_{\tau}\{y\}$ is called the specialization order\cite{11,18,19},
where $cl_{\tau}\{y\}$ is the closure of $\{y\}$. From now on, all order-theoretical statements about $T_{0}$ spaces, such as
upper sets, lower sets, directed sets, and so on, always refer to the specialization order.

A net of a topological space $X$ is a map $\xi: J\rightarrow X$, where $J$ is a directed set. Usually, we denote a net by $(x_{j})_{j\in J}$. Let
$x\in X$, saying $(x_{j})_{j\in J}$ converges to $x$, denoted by $(x_{j})_{j\in J}\rightarrow_{\tau} x$, if $(x_{j})_{j\in J}$ is eventually in every
open neighborhood of $x$, that is, for any given open neighborhood $U$ of $x$, there exists $j_{0}\in J$ such that for every $j\in J$, $j\geq j_{0}\Rightarrow x_{j}\in U$.

Let $X$ be a $T_{0}$ topological space,  then any directed subset of $X$ can be regarded as a net, and its index set is itself. We use $D\rightarrow_{\tau} x$ to represent $D$ converges to $x$. Define notation $D(X)=\{(D, x): x\in X, D$ is a directed subset of $X$ and $D\rightarrow_{\tau} x\}$. It is easy to verify that, for any $x, y\in X$, $x\leq y\Leftrightarrow \{y\}\rightarrow_{\tau} x$ and for directed set $D$, if $\exists d\in D$, $x\leq d$, then $D\rightarrow_{\tau} x$. Therefore, if $x\leq y$, then
$(\{y\}, x)\in D(X)$. Next, we give the concept of directed space. A subset $U$ of $X$ is called a directed open set if $\forall(D, x)\in D(X)$, $x\in U\Rightarrow D\cap U\neq \emptyset$. Denote all directed open sets of $X$ by $d(X)$. Obviously, every open set of $X$ is directed open, that is, $\tau\subseteq d(X)$.

Let $X$ be a $T_{0}$ topological space. $X$ is called directed space if every directed open set of $X$ is an open set, that is, $d(X)=\tau$\cite{13,16}.

\begin{rem}\label{2.1} \cite{13,16} Let $X$ be a $T_{0}$ topological space.

(1) The definition of directed space here is equivalent to the monotone determined space
defined in \cite{17}.

(2) Every poset equipped with the Scott topology is a directed space \cite{16,18}, besides, each
Alexandroff space is a directed space. Thus, the directed space extends the concept of
the Scott topology.

(3) If $U\in d(X)$, $U=\uparrow U$.

(4) $X$ equipped with $d(X)$ is a $T_{0}$ topological space such that $\leq_{d}=\leq$, where $\leq_{d}$ is the specialization
order relative to $d(X)$.

(5) For a directed subset $D$ of $X$, $D\rightarrow x\Leftrightarrow D\rightarrow_{d(X)} x$ for all $x\in X$, where $D\rightarrow_{d(X)} x$ means that $D$ converges to $x$ with respect to the topology $d(X)$.

(6)For each $x\in X$, $\downarrow x$ is directed closed.

\end{rem}
Let $X$ be a $T_{0}$ topological space. The topology
generated by the complements of all principal filters $\uparrow x$ is called the lower topology and denoted by $\omega(X)$. The common refinement $d(x) \vee \omega(X)$ of the directed topology
and lower topology is called the Lawson topology, denoted by $\lambda(X)$. If $X$ be a directed space, then $\lambda(X)=d(x) \vee \omega(X)=\tau \vee \omega(X)$.

Let $X$ be a directed space and $x, y\in X$. We say that $x$ is d-way below $y$, denoted by $x\ll_{d} y$, if for any directed subset $D$ of $X$, $D\rightarrow_{\tau} y$ implies $x\leq d$ for some $d\in D$. If $x\ll_{d} x$, then $x$ is called d-compact element of $X$. For any directed space $X$ and $x\in X$, we denote $K_{d}(X)=\{x\in X: x\ll_{d} x\}$, $\Da_{d} x=\{y\in X: y\ll_{d} x\}$ and $\Ua_{d} x=\{y\in X: x\ll_{d} y\}$. The d-way below relation is a natural extension of way-below relation. Let $X$ be a directed space. $X$ is called d-continuous if $\Da_{d} x$ is directed and $\Da_{d} x\rightarrow_{\tau} x$ for any $x\in X$\cite{19}.

Let $X$ be a $T_{0}$ topological space and $\mathcal{P}^{w}(X)$ be the family of all nonempty finite subsets of $X$. Given any two subsets $G, H$ of $X$, we define $G\leq H$ iff $H\subseteq \uparrow G$. A family of finite sets $\mathcal{F}\subseteq\mathcal{P}^{w}(X)$ is said to be directed if given $F_{1}, F_{2}$ in the family, there exists $F\in \mathcal{F}$ such that $F\subseteq \uparrow F_{1}\cap \uparrow F_{2}$. Unless otherwise specified, directed family is always defined as above.
Let $\mathcal{F}\subseteq \mathcal{P}^{w}(X)$ be a directed family. We say that $\mathcal{F}\rightarrow_{\tau} x$ if for any open neighbourhood $U$ of $x$, there exists some $F\in \mathcal{F}$ such that $F\subseteq U$.

Let $X$ be a directed space and $G, H\subseteq X$. We say that $G$ d-approximates $H$, denoted by $G\ll_{d} H$, if for any directed subset $D$ of $X$, $D\rightarrow h$ for some $h\in H$ implies $D\cap \uparrow G\neq \emptyset$. We
write $G\ll_{d} x$ for $G\ll_{d} \{x\}$. $G$ is said to be d-compact if $G\ll_{d} G$. For any $T_{0}$ topological space $X$ and $F\subseteq X$, we denote $\Uparrow_{d} F=\{x\in X: F\ll_{d} x\}$. A topological space $X$ is called d-quasicontinuous if it is a directed space such that
for any $x\in X$, the family $fin_{d}(x)=\{F: F\in \mathcal{P}^{w}(X), F\ll_{d} x\}$ is a directed family and converges to $x$.

\begin{thm}\label{2.2} \cite{21} Let $X$ be a
d-quasicontinuous space. The following statements hold.

(1) Given any $H\in \mathcal{P}^{w}(X)$ and $y\in X$, $H\ll_{d}y$ implies $H\ll_{d}F\ll_{d}y$ for some finite subset $F\in \mathcal{P}^{w}(X)$.

(2) Given any $F\in \mathcal{P}^{w}(X)$, $\Uparrow_{d} F =(\uparrow F)^{\circ}$. Moreover, $\{\Uparrow_{d} F : F\in \mathcal{P}^{w}(X)\}$ is a base of $\tau$.
\end{thm}

Let $X$ be a d-quasicontinuous space and $x\in X$, then $\Uparrow _{d}\{x\}=\Ua_{d} x$ is a open subset of $X$. A $T_{0}$ topological space $X$ is called locally hypercompact, if for any open subsets $U$ of $X$ and $x\in U$, there exists some $F\in \mathcal{P}^{w}(X)$ such that $x\in (\uparrow F)^{\circ} \subseteq \uparrow F\subseteq U$.

\begin{thm}\label{2.3} \cite{21} Let $X$ be a directed space, Then the following three conditions are equivalent to each other.

(1) $X$ is a quasicontinuous space;

(2) $X$ is a locally hypercompact space;

(3) There exists a directed family $\mathcal{F }\subseteq fin(x)$ such that $\mathcal{F}\rightarrow_{\tau} x$ for any $x\in X$.
\end{thm}

\begin{lem}\label{2.4} \cite{21} Let $X$ be a $T_{0}$ topological space, $\mathcal{F}\subseteq \mathcal{P}^{w}(X)$ be directed family, $G, H\in \mathcal{P}^{w}(X)$, $x\in X$. If $G\ll _{d} H$ and $\mathcal{F}\rightarrow_{\tau} x\in H$, then there exists $F\in \mathcal{F}$ such that $F\subseteq \uparrow G$.
\end{lem}
Continuous dcpos and quasicontinuous dcpos endowed with the Scott topology can be viewed as special d-continuous spaces and d-quasicontinuous spaces.

\begin{defn}\label{2.5} \cite{21} Suppose $X, Y$ are two $T_{0}$ topological spaces. A function $f : X\rightarrow Y$ is called directed
continuous if it is monotone and preserves all limits of directed subset of $X$; that is, $(D, x)\in D(X)\Rightarrow (f(D), f(x))\in D(Y)$.
\end{defn}
Here are some characterizations of the directed continuous functions.

\begin{prop}\label{2.6} \cite{21}  Suppose $X, Y$ are two $T_{0}$ topological spaces. $f : X\rightarrow Y$ is a function between $X$
and $Y$. Then

(1) $f$ is directed continuous if and only if $\forall U \in d(Y)$, $f^{-1}(U)\in d(X)$.

(2) If $X, Y$ are directed spaces, then $f$ is continuous if and only if it is directed continuous.
\end{prop}

An approximate identity for a directed space $X$ is a directed set $\mathcal{D}\subseteq [X\rightarrow X]$ satisfing
$\mathcal{D}\rightarrow 1_{X}$, that is, $\forall x\in X$, $\mathcal{D}(x)\rightarrow_{\tau} x$(pointwise convergence), where [$X\rightarrow X]$ is the family of all continuous functions on $X$.

\begin{defn}\label{2.7} \cite{13} A continuous function $\delta: X\rightarrow X$ on a directed space $X$ is finitely separating if there
exists a finite set $F_{\delta}$ such that for each $x\in X$, there exists $y\in F_{\delta}$ such that $\delta(x)\leq y\leq x$. A directed
space is finitely separated if there is an approximate identity for $X$ consisting of finitely separating
functions. A finitely separated directed space that is also a c-space will be called an FS-space.
\end{defn}
In \cite{13}, Xie and Kou showed that a finitely separated space is already a c-space, hence an FS-space.

\begin{defn}\label{2.8} Let $X$ be a directed space, $\mathcal{B}\subseteq \mathcal{P}^{w}(X)$.

(1) $\mathcal{B}$ is called a quasibase of $X$ if $B\cap fin(x)$ is a directed family and $B\cap fin(x)\rightarrow_{\tau} x$.
(2) A quasibase is called countable if $|\mathcal{B}|$ is countable.
\end{defn}

Let $X$ be a directed space, then $X$ is d-quasicontinuous spaces iff $X$ has a quasibase.

\section{QFS-space}

In this section, two technical notions of quasi-finitely separating map and quasi-approximate identity are introduced. With these notions, we define QFS-spaces and study basic properties and constructions of them.

\begin{defn}\label{3.1}Let $X$ be a directed space. A map $\delta :X\rightarrow \mathcal{P}^{w}(X)$ is said to be quasi-finitely separating if it satisfies the following three conditions:

(i) If $a\leq b$, then $\delta(b)\subseteq \uparrow \delta(a)$;

(ii) There exists a finite set $F_{\delta}\subseteq X$ (called a quasi-finitely separating set) such that for any $x\in X$, there exists $y\in F_{\delta}$ with $x\in \uparrow y\subseteq \uparrow \delta(x)$.

(iii) For each directed set $D\subseteq X$ and $D\rightarrow_{\tau} x$, $\delta(D)\Rightarrow_{\tau} \delta(x)$, thai is, if $\delta(x)\subseteq U\in \tau$,  $\delta(d)\subseteq U$ for some $d\in D$.
\end{defn}

Denote all maps from $X$ to ${P}^{w}(X)$ by $[X\rightarrow {P}^{w}(X)]$. $\mathcal{D}\subseteq [X\rightarrow {P}^{w}(X)]$ is called directed if $\delta_{1}, \delta_{2}\in \mathcal{D}$, $\delta(x)\subseteq \uparrow\delta_{1}(x)\cap \uparrow\delta_{2}(x)$ for any $x\in X$.
We will use $1^{*}_{X}$ to denote the map $1^{*}_{X}: P\rightarrow \mathcal{P}^{w}(X)$ with $1^{*}_{X}(x)=\{x\}$ for any $x\in X$. A quasi-approximate identity for a directed space $X$ is a directed set $\mathcal{D}\subseteq [X\rightarrow {P}^{w}(X)]$ satisfing $\mathcal{D}\rightarrow_{\tau}  1^{*}_{X}$, that is, $\forall x\in X$, $\mathcal{D}(x)\rightarrow_{\tau} x$ (pointwise convergence).

\begin{defn}\label{3.2} A directed space $X$ is called QFS-domain if there is a quasi-approximate identity for $X$ consisting of quasi-finitely separating maps.
\end{defn}
Let $X$ be a QFS-space and $\delta$ be any quasi-finitely separating map on $X$ with quasi-finitely separating set $F_{\delta}$. Since every $x\in X$ is above some $y\in F_{\delta}$ by Definition \ref{3.1}, $X =\uparrow F_{\delta}$ . Since $F_{\delta}$ is finite, QFS-spaces is a finitely generated upper set.

Let $X$ be a finite set and $\mathcal{D}=\{1^{*}_{X}\}$, then $1^{*}_{X}$ is quasi-finitely separating map, $\mathcal{D}$ is directed and $\mathcal{D}\rightarrow  1^{*}_{X}$. Then $X$ is a QFS-space.

\begin{prop}\label{3.3}  Let $X$ be a directed space. If $\delta$ is a quasi-finitely separating map on $X$, then for any $x\in X$, $\delta(x)\ll_{d} x$.
\end{prop}

\begin{proof}\quad Suppose that $x\in X$, directed set $D\subseteq X$  and $D\rightarrow_{\tau} x$. Let $F_{D}=\{y\in F_{\delta}: \exists d\in D, d\in \uparrow y\subseteq \uparrow \delta(d)\}$. For any $y\in F_{D}$, we can choice a $d_{y}\in D$ such that $d_{y}\in \uparrow y\subseteq \uparrow \delta(d_{y})$. Let $D_{F}=\{d_{y}: y\in F_{D}\}$, then $D_{F}$ is finite. Since $D$ is directed set, $D_{F}\subseteq \downarrow d_{0}$ for some $d_{0}\in D$. Then $d_{0}\in \uparrow\delta(d)$ for any $d\in D$.

If $d_{0}\notin \uparrow \delta(x)$, then $\delta(x)\subseteq X\backslash\downarrow d_{0}$. Since $X\backslash\downarrow d_{0}$ is directed open set and $\delta(D)\Rightarrow_{\tau} \delta(x)$, $\delta(d)\subseteq X\backslash\downarrow d_{0}$ for some $d\in D$. Contradiction with $d_{0}\in \uparrow\delta(d)$. Hence $d_{0}\in \uparrow \delta(x)$ and $\delta(x)\ll_{d} x$.
\end{proof}

\begin{prop}\label{3.4}  Let $X$ be a FS-space, then $X$ is QFS-space.
\end{prop}
\begin{proof} Let $X$ be a FS-space and $\mathcal{D}$ be an approximate identity for $X$ consisting of finitely
separating functions, let $\mathcal{D}^{*}=\{\delta^{*}: \delta\in \mathcal{D}\}$, where $\delta^{*}(x)=\{\delta(x)\}$ for any $\delta\in \mathcal{D}$ and $x\in X$.
Then $\mathcal{D}^{*}$ is a quasi-approximate identity for $X$ consisting of quasi-finitely separating
maps.
\end{proof}
\begin{prop}\label{3.5}  Let $X$ be a QFS-space and $\mathcal{D}$ be a quasi-approximate identity consisting of
quasi-finitely separating maps for $X$, then $X$ is d-quasicontinuous space. In particular, $\{\delta(x): \delta\in \mathcal{D}, x\in X\}$ is quasibase of $X$.
\end{prop}
\begin{proof}  Let $X$ be a QFS-space and $\mathcal{D}$ be a quasi-approximate identity for $X$ consisting
of quasi-finitely separating maps. Since $\mathcal{D}(x)=\{\delta(x)\}_{\delta\in \mathcal{D}}$ is directed and $\{\delta(x)\}_{\delta\in \mathcal{D}}\rightarrow_{\tau} x$, $X$ is d-quasicontinuous space by $fin_{d}(x)\rightarrow x$ and Theorem \ref{2.3}.

By Proposition \ref{3.3}, $\{\delta(x): \delta\in \mathcal{D}, x\in X\}$ is quasibase of $X$.
\end{proof}
\begin{lem}\label{3.6}  Let $X$ be a QFS-space and $\mathcal{D}$ be a quasi-approximate identity consisting of
quasi-finitely separating maps for $X$. For any open subset $U\subseteq X$, let $V_{\delta}=\{x\in X: \uparrow \delta (x)\subseteq U\}$ for any $\delta \in \mathcal{D}$. Then $V_{\delta}\in d(X)$, $\bigcup_{\delta\in \mathcal{D}}V_{\delta}=U$ and $\{V_{\delta}: \delta \in \mathcal{D}\}$ is directed family respect to inclusion order.
\end{lem}
\begin{proof}  First we show that $\bigcup_{\delta\in \mathcal{D}}V_{\delta}=U$. For any $x\in \bigcup_{\delta\in \mathcal{D}}V_{\delta}$, there exists $\delta\in \mathcal{D}$ such that $\uparrow \delta(x)\subseteq U$. Since $\delta$ is quasi-finitely separating map, $x\in \uparrow \delta(x)\subseteq U$. Conversely, for any $x\in U$, since $\mathcal{D}$ is a quasi-approximate identity, $\mathcal{D}(x)\rightarrow_{\tau} x$. Then there exists $\delta_{0}\in \mathcal{D}$ such that $\uparrow \delta_{0}(X)\subseteq U$. Thus $x\in V_{\delta_{0}}\subseteq\bigcup_{\delta\in \mathcal{D}}V_{\delta}$.

Then we show that $V_{\delta}$ is open set for any $\delta\in \mathcal{D}$. Let $x\in V_{\delta}$ and $x\leq y$. Since $\delta$ is quasi-finitely separating, $\delta(y)\subseteq \uparrow \delta(x)\subseteq U$, $y\in V_{\delta}$. Thus  $V_{\delta}$ is an upper set. For any directed subsets $D\subseteq L$ with
$D\rightarrow_{\tau} x\in V_{\delta}$. By $\delta(D)\Rightarrow_{\tau} \delta(x)\subseteq U$, $\delta(d)\subseteq U$ for some $d\in D$. Then $\uparrow \delta(d)\subseteq U$, $d\in V_{\delta}$. Hence $V_{\delta}$ is open set.

Let $V_{\delta_{1}}, V_{\delta_{2}}\in \{V_{\delta}: \delta \in \mathcal{D}\}$. Since $\mathcal{D}$ is directed, $\delta(x)\subseteq \uparrow\delta_{1}(x)\cap \delta_{2}(x)$ for any $x\in X$. If $x\in V_{\delta_{1}}$, then $\uparrow \delta(x)\subseteq \uparrow \delta_{1}(x)\subseteq U$ and $x\in V_{\delta}$. Hence $V_{\delta_{1}}\subseteq V_{\delta}$. Similarly, $V_{\delta_{2}}\subseteq V_{\delta}$. $\{V_{\delta}: \delta \in \mathcal{D}\}$ is directed family respect to inclusion order.
\end{proof}

\begin{lem}\label{3.7} Let $X$ be a QFS-space and $\mathcal{D}$ be a quasi-approximate identity consisting of
quasi-finitely separating maps for $X$. For any $F\in \mathcal{P}^{w}(X)$, $\uparrow F=\bigcap _{\delta\in \mathcal{D}}\bigcup_{x\in F}\uparrow \delta (x)$ and $\{\bigcup_{x\in F}\uparrow \delta (x):\delta\in \mathcal{D}\}$ is directed family.
\end{lem}
\begin{proof} Obviously, $\uparrow F\subseteq\bigcap _{\delta\in \mathcal{D}}\bigcup_{x\in F}\uparrow \delta (x)$. Conversely, if there exists $y\in\bigcap _{\delta\in \mathcal{D}}\bigcup_{x\in F}\uparrow \delta (x)$
but $y\notin \uparrow F$, then $\uparrow F\subseteq X\backslash \downarrow y\in \tau$. Then $X\backslash \downarrow y=\bigcup_{\delta\in \mathcal{D}}\{x\in X: \uparrow \delta (X)\subseteq X\backslash \downarrow y\}$ by Lemma \ref{3.6}. Since $\uparrow F$ is compact, there exists $\delta_{i}\subseteq \mathcal{D}$, $i=1, \cdots, n$, such that $\uparrow F\subseteq \bigcup_{i=1, \cdots, n}\{x\in X: \uparrow \delta_{i} (x)\subseteq X\backslash \downarrow y\}$. Since $\{V_{\delta}: \delta \in \mathcal{D}\}$ is directed family, there exists $\delta_{0}\in \mathcal{D}$ such that $\uparrow F\subseteq \{x\in X: \uparrow \delta_{0} (x)\subseteq X\backslash \downarrow y\}$ which implies $y\notin\bigcup_{x\in F}\uparrow \delta_{0} (x)$. Hence $y\in \uparrow F$.

Thus $\uparrow F=\bigcap _{\delta\in \mathcal{D}}\bigcup_{x\in F}\uparrow \delta (x)$. Obviously, $\{\bigcup_{x\in F}\uparrow \delta (x):\delta\in \mathcal{D}\}$ is directed family.
\end{proof}

\begin{defn}\label{3.8} A family $\mathcal{F}$ of maps from $X$ to $\mathcal{P}^{w}(X)$ is said to satisfy property $\mathcal{P}$ if given
any  $F_{i}\ll_{d}x_{i}$, $i=1,2$, there is a $\delta \in \mathcal{F}$ such that $F_{i}\ll_{d}\delta(x_{i})\ll_{d}x_{i}$.
\end{defn}

\begin{prop}\label{3.9} Let $X$ be a QFS-space and $\mathcal{D}$ be a quasi-approximate identity consisting of
quasi-finitely separating maps for $X$, Then $X$ is a d-quasicontinuous space and $\mathcal{D}$ satisfies property $\mathcal{P}$.
\end{prop}
\begin{proof} Suppose $X$ is a QFS-space, then $X$ is d-quasicontinuous space by Proposition \ref{3.5}. Given any pairs $F_{i}\ll_{d}x_{i}$, $i=1,2$. Let $\mathcal{D}$ be a quasi-approximate identity for $X$ consisting of quasi-finitely separating maps. For any $i$, $1\leq i\leq 2$, $\mathcal{D}(x_{i})\rightarrow_{\tau} x_{i}\in \Uparrow F_{i}\in \tau$.
Then there exists $\delta_{0}\in \mathcal{D}$ such that $\delta_{0}(x_{i})\subseteq \Uparrow F_{i}$. Thus $F_{i}\ll_{d}\delta_{0}(x_{i})\ll_{d}x_{i}$ by Proposition \ref{3.3}. $\mathcal{D}$ satisfies property $\mathcal{P}$.
\end{proof}

\section{Properties of QFS-spaces}

In this section, we show that closed subspaces ,quasicontinuous projection spaces and continuous retracts of QFS-spaces are QFS-spaces. Further more, a kind of continuous retracts of QFS-space are constructed.

Let $X=(X, \tau)$ be a topological space, $A\subseteq X$, then $A=(A,\tau|_{A})$ is a subspace of $X$. For $x,y\in A$, it is easy to see that $x\leq y$ iff $x\leq_{\tau|_{A}}y$, where $\leq_{\tau|_{A}}$ is the specialization
order relative to $\tau|_{A}$.

\begin{thm}\label{4.1} Let $X$ be a QFS-space and $A\subseteq X$ be a directed closed set. Then $A$ is a QFS-space.
\end{thm}
\begin{proof} Since $X$ is a QFS-domain, there exists a quasi-approximate identity $\mathcal{D}$ for
$X$ consisting of quasi-finitely separating maps. Let $\mathcal{D}^{*} = \{\delta_{A}: \delta\in \mathcal{D}\}$, where $\delta_{A}(x)=\delta(x)\cap A$ for any $x\in A$. Now we prove it in three steps.

(1) Let $\Gamma(X)$ be set of all directed closed subsets of $X$ and $\Gamma(A)$ be set of all directed closed subsets of $A$. Since $A$ is directed set, $\Gamma(A)=\{G\in \Gamma(A): G\subseteq A$.  Hence $d(A)=\{A\setminus G: G\in \Gamma (A)\}=\{A\setminus G: G\in \Gamma (X)\mid _{A}\}=\{A\setminus A\cap G: G\in \Gamma (X)\}=\{A\setminus G: G\in \Gamma (X)\}=\{A\cap X\setminus G: G\in \Gamma (X)\}=\{A\cap U: U\in d (X)\}=d(X)|_{A}$. $A$ is directed space.

(2) we verify that $\mathcal{D}^{*}$ satisfies the conditions in Definition \ref{3.1}. Let $X=(X,\tau)$, $A=(A, \tau|_{A})$.

(i) For any $a,b\in A$ with $a\leq b$, we have $\delta_{A}(b)=\delta(b)\cap A\subseteq \uparrow \delta(a)\cap A=\uparrow \delta_{A}(a)$.

(ii)For any $\delta\in \mathcal{D}$, there exists a finite set $F_{\delta}\subseteq X$ such that
for any $a\in A$, there exists $y\in F_{\delta}$ with $a\in \uparrow y\subseteq \uparrow \delta(a)$. By closedness of $A$, we have $y\in A$. Then $A\cap F_{\delta}\neq\emptyset$. Let $F_{\delta_{A}}=A\cap F_{\delta}$. Then
for any $a\in A$, there exists $y\in F_{\delta_{A}}\subseteq A$ such that $a\in \uparrow_{A}y =\uparrow y \cap A\subseteq \uparrow \delta(a)\cap A=\uparrow _{A}\delta_{A}(a)$.

(iii) Let $D$ be directed subset of $A$ and $D\rightarrow_{\tau|_{A}} x\in A$, then $D$ is directed subset of $X$ and $D\rightarrow_{\tau} x$. Hence $\delta(D)\Rightarrow_{\tau} \delta(x)$. Let $\delta_{A}(x)\subseteq A\cap U$, where $U\in d(X)$. If $a\notin \delta_{A}(x)$ and $a\in \delta(x)$, then $a\notin A$. Hence $a\in X\backslash A$. Since $X$ is quasicontinuous space, there exists $F_{a}\in \mathcal{P}^{w}(X)$ such that $F_{a}\ll a$ and $F_{a}\subseteq X\backslash A$. Let $M=\{a\in \delta(x):a\notin \delta_{A}(x)\}$ and $V=U\cup (\bigcup_{a\in M}\Uparrow F_{a})$, then $\delta(x)\subseteq V$ and $A\cap V=A\cap U$. By $\delta(D)\Rightarrow_{\tau} \delta(x)$ and $V\in d(X)$, $\delta(d)\subseteq V$ for some $d\in D$. Then $\delta(d)\cap A=\delta_{A}(d)\subseteq A\cap V=A\cap U$, hence $\delta_{A}(D)\Rightarrow_{\tau|_{A}} \delta_{A}(x)$.

By Definition \ref{3.1}, $\mathcal{D}^{*}$ is a set consisting of quasi-finitely separating maps on $A$.

(3) Then we show that $\mathcal{D}^{*}$ is a quasi-approximate identity for $A$. It is easy to see
that $\mathcal{D}^{*}$ is directed. For each $a\in U\in \tau|_{A}$, there exists $V\in \tau$ such that $U=A\cap V$. Since $a\in V$ and $\mathcal{D}(a)=\{\delta(a):\delta\in \mathcal{D}\}\rightarrow_{\tau} a$, $\delta^{0}(a)\subseteq V$ for some $\delta^{0}\in \mathcal{D}$. Then $\delta^{0}_{A}(a)=\delta^{0}(a)\cap A\subseteq A\cap V=U$. Hence $\{\delta_{A}(a):\delta\in \mathcal{D}\}\rightarrow_{\tau|_{A}} a$. This shows that $\mathcal{D}^{*}$ is a quasi-approximate identity for $A$.

By (1), (2) and (3), $A$ is a QFS-space.
\end{proof}

\begin{defn}\label{4.2} A function $f:X\rightarrow Y$ between directed space is quasicontinuous if $f$ is monotone, that is, if $H\subseteq \uparrow G$, then $f(H)\subseteq \uparrow f(G)$ and If $\mathcal{F}\subseteq \mathcal{P}^{w}(X)$ and $\mathcal{F}\Rightarrow_{\tau} H$, then $\{f(F): F\in \mathcal{F}\}\Rightarrow_{\tau} f(H)$.
\end{defn}
If $f:X\rightarrow Y$ is a quasicontinuous, then $f$ is a continuous.

\begin{prop}\label{4.3}  If A function $f:X\rightarrow Y$ between directed space is quasicontinuous, then

(1) $f$ is a continuous;

(2) If $\mathcal{F}\rightarrow_{\tau} x$, then $\{f(F): F\in \mathcal{F}\}\rightarrow_{\tau} f(x)$.
\end{prop}

\begin{thm}\label{4.4} Let $X$ be a QFS-space, $f: X\rightarrow X$ be a quasicontinuous projection and $f(X)=(f(X), \tau|_{f(x)})$ be directed space. Then $f(X)$ is a QFS-space.
\end{thm}
\begin{proof} Since $X$ is a QFS-domain, there exists a quasi-approximate identity $\mathcal{D}$ for
$X$ consisting of quasi-finitely separating maps. Let $\mathcal{D}^{*} = \{f\circ \delta: \delta\in \mathcal{D}\}$, where $f\circ \delta(x)=f(\delta(x))$ for any $x\in X$. Firstly, we
verify that $\mathcal{D}^{*}$ satisfies the conditions in Definition \ref{3.1}. Let $X=(X,\tau)$, $f(X)=(f(X), \tau|_{f(X)})$.

(i) For any $a,b\in f(X)$ with $a\leq b$, we have $f\circ \delta(b) = f(\delta(b))\subseteq \uparrow f(\delta(a))=\uparrow f \circ\delta(a)$ for any $\delta \in \mathcal{D}$.

(ii)For any $\delta\in \mathcal{D}$, there exists a finite set $F_{\delta}\subseteq X$ such that
for any $x\in X$, there exists $y\in F_{\delta}$ with $x\in \uparrow y\subseteq \uparrow \delta(x)$. Noticing that $f$ is a projection, then for every $x\in f(X)$, we have $f(x) = x$. By the quasicontinuity of $f$ and $f(y)\in f(F_{\delta})$, we have that $x\in \uparrow_{f(X)} f(\uparrow y)=f(\uparrow y)\cap f(X)\subseteq f(\uparrow \delta(x))\cap f(X)=\uparrow_{f(X)} f\circ \delta(x)$. Hence , $\{f(y):y\in F_{\delta}\}$ is quasi-finitely separating set of $f\circ\delta$.

(iii) Let $D$ be a directed set of $f(X)$ and $D\rightarrow_{\tau|_{A}} x$, then $D$ is directed set of $X$ and $D\rightarrow_{\tau} x$. Hence $\delta (D)\Rightarrow_{\tau} \delta(x)$. Since $f: X\rightarrow X$ is quasicontinuous, $f\circ\delta (D)\Rightarrow_{\tau} f\circ\delta(x)$.

By Definition \ref{3.1}, $\mathcal{D}^{*}$ is a set consisting of quasi-finitely separating maps on $f(X)$.

Then we show that $\mathcal{D}^{*}$ is a quasi-approximate identity for $f(X)$. It is easy to see
that $\mathcal{D}^{*}$ is directed. For each $x\in U\in \tau|_{f(X)}$, there exists $V\in \tau$ such that $U=f(X)\cap V$. Since $x\in V$, $\mathcal{D}(x)=\{\delta(x):\delta\in \mathcal{D}\}\rightarrow_{\tau} x$. Since $f$ is quasicontinuous, $\{f(\delta(x)):\delta\in \mathcal{D}\}\rightarrow_{\tau} f(x)$. Hence $\{f\circ\delta(x):\delta\in \mathcal{D}\}\rightarrow_{\tau|_{f(X)}} x$ by $f(x)=x$.

This shows that $\mathcal{D}^{*}$ is a quasi-approximate identity for $f(X)$. So, $f(X)$ is a QFS-space.
\end{proof}
\begin{defn}\label{4.5} Let $X, Y$ be two topological space. $Y$ is called a continuous retract of $X$ if
there exist continuous functions $f: X\rightarrow Y$ and $g :Y\rightarrow X$ such that $f\circ g = 1_{Y}$.
\end{defn}
\begin{thm}\label{4.6} Let $X$ be a QFS-sapce and $Y$ be a directed space. If $Y$ is a continuous retract of $X$, then $Y$ is a QFS-spaces.
\end{thm}
\begin{proof} Let $X$ be a QFS-space and $Y$ be a retract of $X$. Then there exist continuous
functions $f: X\rightarrow Y$ and $g :Y\rightarrow X$ such that $f\circ g = 1_{Y}$. We show that $Y$ is a QFS-space.
Suppose $\mathcal{D}$ is a quasi-approximate identity for $X$ consisting of quasi-finitely separating
maps. Let $\varepsilon: Y\rightarrow \mathcal{P}^{w}(Y)$ be $\varepsilon(y)=\{f(x): x\in \delta(g(y))\}$ for any $y\in Y$. We show that
$\mathcal{D}^{*}=\{\varepsilon: \delta\in \mathcal{D}\}$ is a quasi-approximate identity for $Y$ consisting of quasi-finitely separating maps.

(1) $\varepsilon$ is quasi-finitely separating map.

(i) For any $y_{1}, y_{2}\in Y$ and $y_{1}\leq y_{2}$. Since $g$ is monotone, $g(y_{1})\leq g(y_{2})$. Hence
$\varepsilon(y_{2})=\{f(x): x\in \delta(g(y_{2}))\} \subseteq\{f(x): x\in \delta(g(y_{1}))\}=\varepsilon(y_{1})$. So $\delta(g(y_{2}))=\{f(x): x\in \delta(g(y_{2}))\}\subseteq \{f(x): x\in \uparrow\delta(g(y_{1}))\}=\uparrow\{f(x): x\in \delta(g(y_{1}))\}$.

(ii)Since $\delta$ is quasi-finitely separating, there exists a finite set $F_{\delta}\subseteq X$ such that
for any $x\in X$, there exists $y\in F_{\delta}$ with $x\in \uparrow y\subseteq \uparrow \delta(x)$. Let $F_{\varepsilon}=f(F_{\delta})$. For any $y\in Y$, we
have $f(g(y)) = y$. Since $g(y)\in X$, there exists $z\in F_{\delta}$ such that $g(y) \in \uparrow z\subseteq \uparrow \delta(g(y))$, then $y\in \uparrow f(z)\subseteq \{f(x): x\in \delta(g(y))\}=\uparrow\varepsilon(y)$.

(iii)For any directed set $D\subseteq Y$ and $D\rightarrow_{\tau_{Y}} y$, then $g(D)\subseteq X$ is directed set and $g(D)\rightarrow_{\tau_{X}} g(y)$. Then $\delta(g(D))\Rightarrow_{\tau_{X}} \delta(g(y))$. Let $\varepsilon(y)=\{f(x): x\in \delta(g(y))\}\in V\in \tau_{Y}$, then $\delta(g(y))\in f^{-1}(V)$. By Definition \ref{3.1} and $\delta(g(D))\Rightarrow_{\tau_{X}} \delta(g(y))$, $\delta(g(d))\subseteq f^{-1}(V)$. That is $\varepsilon(d)=\{f(x): x\in \delta(g(d))\}\in V$, then $\varepsilon(D)\Rightarrow_{\tau_{Y}} \varepsilon(y)$.

By Definition \ref{3.1}, $\mathcal{D}^{*}$ is a set consisting of quasi-finitely separating maps on $Y$.

(2) Then we show that $\mathcal{D}^{*}$ is a quasi-approximate identity for $f(X)$. $\mathcal{D}^{*}$ is directed since $\mathcal{D}$ is directed. $\forall y\in Y$, $\mathcal{D}^{*}(y)=\{\varepsilon(y):\varepsilon\in \mathcal{D}^{*}\}=\{\{f(x): x\in \delta g(y)\}: \delta\in \mathcal{D}\}$. Let $y\in V\in \tau_{Y}$, then $f^{-1}(y)\in f^{-1}(V)\in \tau_{X}$. Since $\mathcal{D}(f^{-1}(y))\rightarrow_{\tau_{X}} f^{-1}(y)$, there exist $\delta \in \mathcal{D}$ such that $\delta (f^{-1}(y))\subseteq f^{-1}(V)$. Then $\delta (f^{-1}(y))= \delta ((g((y)))\subseteq g(V)$. Hence, $\varepsilon(y)=\{f(x): x\in \delta(g(y))\}\subseteq V$, that is $\varepsilon(y)\subseteq V$. Hence $\mathcal{D}^{*}(y)\rightarrow_{\tau_{Y}} x$.

This shows that $\mathcal{D}^{*}$ is a quasi-approximate identity for $f(X)$. So, $f(X)$ is a QFS-space.
\end{proof}

\begin{prop}\label{4.7} Let $X$ be a QFS-space, $A\in X$ be a lower bound of $\Uparrow b$ for some $b\in X$. Then $\{a\}\cup \Uparrow b$ is a continuous  retracts of $X$, that is,  $\{a\}\cup \Ua_{d} b$ is QFS-space.
\end{prop}

\begin{proof} Let $A=\{a\}\cup \Ua_{d} b$ with $\tau|_{A}$, then $g: \{a\}\cup \Ua_{d} b\rightarrow X$, defined by $g(x)=x$. Let $f:X\rightarrow \{a\}\cup \Ua_{d} b$, defined by $f(x)=x$, $x\in A$ or $f(x)=a$, $x\notin A$. It is easy to see that $g$ is continuous.

Obviously, $f$ is monotone and $f=f^{2}$. Let $D\subseteq X$ be directed set and $D\rightarrow x$. (1) If $f(x)=x\in \Ua_{d} b$, then there exists $d_{1}\in D\cap \Ua_{d} b$ such that $b\ll d_{1}$. If $x\in A\cap U$ for some $U\in \tau$, then there exists $d_{2}\in D\cap U$ such that $d_{2}\in U$. Since $D$ is directed set, there exists $d_{0}\in D$ such that $d_{1},d_{2}\leq d_{0}$. Hence $d_{0}\in A\cap U\cap f(D)$, $f(D)\rightarrow_{\tau|_{A}} x$. (2) If $f(x)=a$, then $f(d)\geq f(x)=a$, $f(D)\rightarrow_{\tau|_{A}} a$. Hence, $f$ is continuous. Obviously, $f\circ g=1_{A}$, $A$ is a continuous retract of $X$.

Let $V$ be a directed open subset of $A$ and $D\subseteq A$ be a directed set.

(1) Let $D\rightarrow_{\tau|_{A}}x\in V\in d(A)$. If $x=a$, then $V=A$ and $V\in \tau|_{A}$;

(2) Let $D\subseteq X$ be a directed set and $D\rightarrow _{\tau}x$. If $x\in \Ua b$, $D\cap \Ua_{d} b\neq \emptyset$ by $D\rightarrow _{\tau}x$ and $\Ua_{d} b\in \tau$. Let $d\in D\cap \Ua_{d} b$, then $\uparrow d\cap D$ is directed. Let $x\in U\in \tau$, there exists $d'\in D$ such that $d'\in U$. Since $D$ is directed, $d,d'\leq d*$ for some $d*\in D$. Then $d*\in \uparrow d\cap D\cap \Ua_{d} b \cap U$ and $\uparrow d\cap D\rightarrow_{\tau|_{A}} x$. Since $V\in d(A)$, $\uparrow d\cap D\cap V\neq\emptyset$. That is $D\cap V\neq\emptyset$, $V\in d(X)$. Hence $V\in \tau|_{A}$.

By (1) and (2), $A$ is directed space. Hence $A$ is QFS-space by Theorem \ref{4.6}.
\end{proof}
Suppose $X, Y$ are two directed spaces. Let $X\times Y$ represents the Cartesian product of $X$ and $Y$, then
we have a natural partial order on it: $\forall(x_{1}, y_{1}), (x_{2}, y_{2})\in X\times Y$, $(x_{1}, y_{1})\leq (x_{2}, y_{2})\Leftrightarrow x_{1}\leq x_{2}, y_{1}\leq y_{2}$, which is called the pointwise order on $X\times Y$. Now, we define a topological space $X\otimes Y$ as follows\cite{13}:

(1) The underlying set of $X\otimes Y$ is $X\times Y$;

(2) The topology on $X\times Y$ is generated as follows: For each given directed set $D\subseteq X\times Y$ and
$(x, y)\in X\times Y$, $D\rightarrow _{X\oplus Y}(x, y)\in X \otimes Y \Leftrightarrow \pi_{1}D\rightarrow \tau_{X}x \in X$, $\pi_{2}D\rightarrow \tau_{Y}y \in Y$, that is, a subset $U\subseteq X\otimes Y$ is open if and only if for every directed limit defined as above
$D\rightarrow _{X\oplus Y}(x, y)\in U \Rightarrow U\cap D \neq\emptyset$.

\begin{prop}\label{4.8}\cite{13}
Suppose $X$ and $Y$ are two directed spaces.

(1) The topological space $X\otimes Y$ defined above is a directed space and satisfies the following
properties: The specialization order on $X\otimes Y$ equals to the pointwise order on $X\times Y$.

(2) Suppose $Z$ is another directed space, then $f :X\otimes Y\rightarrow Z$ is continuous if and only if it is
continuous in each variable separately.
\end{prop}

\begin{thm}\label{4.9}
Let $X, Y$ be QFS-space and $X$ be core-compact, then $X\otimes Y$ is QFS-space.
\end{thm}
\begin{proof}
Let $X, Y$ be QFS-space and. Suppose $\mathcal{D}$, $\mathcal{E}$ be quasi-approximate identity for $X$, $Y$ consisting of quasi-finitely separating
maps. Let $\mathcal{D}*=\{(\delta, \varepsilon):\delta\in \mathcal{D}, \varepsilon\in \mathcal{E}\}$. Since $X$ is core-compact, $X\otimes Y = X \times Y$ by Theorem 4.2 in \cite{13}.

It is easy to verify that $\mathcal{D}*$ is quasi-approximate identity for $X\otimes Y$ consisting of quasi-finitely separating
maps.
\end{proof}

\section{QFS-space and Convex Powerspace of directed space}

\begin{defn}\cite{25}\label{5.1}
Let $X$ be a directed space.

(1) Denote $\mathcal{Q}_{fin}(X)=\{\uparrow F:F\in \mathcal{P}^{w}(X)\}$. Define an order $\leq_{\mathcal{Q}}$ on $\mathcal{Q}_{fin}(X)$: $\uparrow F_{1}\leq_{\mathcal{Q}}\uparrow F_{2}\Leftrightarrow \uparrow F_{2}\subseteq \uparrow F_{1}$.

(2)Let $\mathcal{F}=\{\uparrow F: F\in \mathcal{P}^{w}(X)\}\subseteq \mathcal{Q}_{fin}(X)$ be a directed set(respect to order  $\leq_{\mathcal{Q}}$) and $\uparrow H\subseteq\mathcal{Q}_{fin}(X)$. Define a convergence notation
$\mathcal{F}\Rightarrow_{\mathcal{Q}}\uparrow H \Leftrightarrow $there exists finite directed sets $D_{1}, \cdots, D_{n}\subseteq X$ such that

1. $H\cap \{x: D_{i}\rightarrow x\}\neq\emptyset$ for any $i=1 , \cdots, n$;

2. $H\subseteq \{x: D_{i}\rightarrow x, i= , \cdots, n\}$;

3. $\forall (d_{1},\cdots, d_{n})\in \prod^{n}_{1}D_{i}$, there exists some $\uparrow F\in \mathcal{F}$ such that $\uparrow F\in \bigcup^{n}_{1}\uparrow d_{i}$.

(3)A subset $\mathcal{U}\subseteq \mathcal{Q}_{fin}(X)$ is called a $\Rightarrow_{\mathcal{Q}}$ convergence open set of $\mathcal{Q}_{fin}(X)$ if and only if for any directed subset $\mathcal{F}$ of $\mathcal{Q}_{fin}(X)$ and $\uparrow H\in \mathcal{Q}_{fin}(X)$, $\mathcal{F}\Rightarrow_{\mathcal{Q}}\uparrow H\in \mathcal{U}$ implies $\mathcal{F}\cap \mathcal{U}\neq \emptyset$. Denote all $\Rightarrow_{\mathcal{Q}}$ convergence open set of $\mathcal{Q}_{fin}(X)$ by $\mathcal{O}_{\Rightarrow_{\mathcal{Q}}}(\mathcal{Q}_{fin}(X))$.
\end{defn}

\begin{prop}\cite{25}\label{5.2}
Let $X$ be a directed space, then

(1) $(\mathcal{Q}_{fin}(X), \mathcal{O}_{\Rightarrow_{\mathcal{Q}}}(\mathcal{Q}_{fin}(X)))$ is a directed space, that is, $\mathcal{O}_{\Rightarrow_{\mathcal{Q}}}(\mathcal{Q}_{fin}(X))=d(\mathcal{Q}_{fin}(X))$ and the specialization order $\leq$ of $\mathcal{O}_{\Rightarrow_{\mathcal{Q}}}(\mathcal{Q}_{fin}(X))$ equals to $\leq_{\mathcal{Q}}$;

(2) $(\mathcal{Q}_{fin}(X), \mathcal{O}_{\Rightarrow_{\mathcal{Q}}}(\mathcal{Q}_{fin}(X)))$ respect to the set union
operation $\cup$ is a directed deflationary semilattice;

(3)Suppose $X$ is a directed space, then $(\mathcal{Q}_{fin}(X), \mathcal{O}_{\Rightarrow_{\mathcal{Q}}}(\mathcal{Q}_{fin}(X)))$ is the upper powerspace
of $X$, that is, endowed with topology $\mathcal{O}_{\Rightarrow_{\mathcal{Q}}}(\mathcal{Q}_{fin}(X))$, $(\mathcal{Q}_{fin}(X), \cup)\cong P_{U}(X)$.

\end{prop}

Let $X$ be a directed space and $\uparrow G, \uparrow H\in \mathcal{Q}_{fin}(X)$, then $\uparrow G\ll  \uparrow H\Leftrightarrow \forall \mathcal{F}\Rightarrow_{\mathcal{Q}}\uparrow H$, $\uparrow F\subseteq \uparrow G$ for some $\uparrow F\in \mathcal{F}$ by the definition of way below.
In this paper, the way below relation with respect to $\leq\mathcal{_{Q}}$ and $\Rightarrow_{\mathcal{Q}}$ is denoted by $\ll_{\mathcal{Q}}$.
Let $\uparrow G\ll_{\mathcal{Q}} \uparrow H$. Since $\uparrow H\leq_{\mathcal{Q}}\uparrow x$ for any $x\in H$, $\uparrow G\ll_{\mathcal{Q}} \uparrow x$.

The following lemma shows that the way below relation on the poset $\mathcal{Q}_{fin}(X)$ agrees with the way-below relation defined for finite subsets of directed space.

\begin{prop}\label{5.3}
Let $X$ be a directed space and $\uparrow G, \uparrow H\in \mathcal{Q}_{fin}(X)$, then $\uparrow G\ll_{\mathcal{Q}} \uparrow H$ iff $G\ll H$.
\end{prop}

\begin{proof}
If $\uparrow G\ll_{\mathcal{Q}} \uparrow H$, $D$ is a directed subset of $X$ and $D\rightarrow x\in H$. Let $\mathcal{F}=\{\{d\}: d\in D\}$, then (1) $\{x\}\cap \{x: D\rightarrow x\}\neq\emptyset$; (2) $\{x\} \subseteq \{x: D\rightarrow x\}$; (3)$\forall d\in D$, there exists some $F=\{d\}\in \mathcal{F}$ such that $\uparrow F\subseteq \uparrow d$. By the Definition \ref{5.1}, $\mathcal{F}\Rightarrow_{\mathcal{Q}}\uparrow x$. Since $\uparrow G\ll_{\mathcal{Q}} \uparrow x$, $\uparrow F\subseteq \uparrow G$ for some $F\in \mathcal{F}$, that is $F=\{d\}\subseteq \uparrow G$. So, $\uparrow \cap D\neq\emptyset$ and $G\ll x$. By the arbitrariness of $x$, $G\ll H$.

Let $G\ll H$ and $\mathcal{F}\Rightarrow_{\mathcal{Q}}\uparrow H$. Then there exists finite directed sets $D_{1}, \cdots, D_{n}\subseteq X$ such that

1. $H\cap \{x: D_{i}\rightarrow x\}\neq\emptyset$ for any $i=1, \cdots, n$;

2. $H\subseteq \{x: D_{i}\rightarrow x, i= 1, \cdots, n\}$;

3. $\forall (d_{1},\cdots, d_{n})\in \prod^{n}_{1}D_{i}$, there exists some $\uparrow F\in \mathcal{F}$ such that $\uparrow F\in \bigcup^{n}_{1}\uparrow d_{i}$.

For any $i$, let $x_{i}\in H\cap \{x: D_{i}\rightarrow x\}$, then $D_{i}\rightarrow x_{i}\in H$. Thus $D_{i}\cap \uparrow G\neq \emptyset$. Let $d_{i}\in D_{i}\cap \uparrow G$, then there exists some $\uparrow F\in \mathcal{F}$ such that $\uparrow F\in \bigcup^{n}_{1}\uparrow d_{i}\subseteq \uparrow G$, that is $ F\subseteq \uparrow G$. So, $\uparrow G\ll_{\mathcal{Q}}\uparrow H$.
\end{proof}

\begin{prop}\label{5.4}
Let $X$ be a continuous space, $\mathcal{F}\subseteq \mathcal{Q}_{fin}(X)$ and $ H\in\mathcal{P}^{w}(X)$. Then $\mathcal{F}\Rightarrow_{\mathcal{Q}}\uparrow H$ iff $H\subseteq U$ implies there exists $ F\in \mathcal{F}$ such that $F\subseteq \uparrow F\subseteq U$ for any $U\in d(X)$.
\end{prop}

\begin{proof}
Let $\mathcal{F}\Rightarrow_{\mathcal{Q}}\uparrow H$ and $H\subseteq U\in d(X)$, then there exists finite directed sets $D_{1}, \cdots, D_{n}\subseteq X$ such that

1. $H\cap \{x: D_{i}\rightarrow x\}\neq\emptyset$ for any $i=1 , \cdots, n$;

2. $H\subseteq \{x: D_{i}\rightarrow x, i=1 , \cdots, n\}$;

3. $\forall (d_{1},\cdots, d_{n})\in \prod^{n}_{1}D_{i}$, there exists some $ F\in \mathcal{F}$ such that $\uparrow F\in \bigcup^{n}_{1}\uparrow d_{i}$.

Hence for any $i$, $D_{i}\rightarrow h$ for some $h\in H$.  Since $h\in U$, $D_{i}\cap U\neq \emptyset$. Let $d_{i}\in D_{i}\cap U$, then there exists some $\uparrow F\in \mathcal{F}$ such that $\uparrow F\subseteq \bigcup^{n}_{1}\uparrow d_{i}\subseteq U$, that is, $F\subseteq \uparrow F\subseteq U$.

If $H\subseteq U$ implies there exists $\uparrow F\in \mathcal{F}$ such that $F\subseteq\uparrow F\subseteq U$ for any $U\in d(X)$. For any $h\in H$, since $X$ is continuous space, $\Da h$ is directed set and converges to $x$. Let $H=\{h_{1},\cdots ,h_{n}\}$ and $D_{i}=\Da h_{i}$. Obviously, $H\cap \{x: D_{i}\rightarrow x\}\neq\emptyset$ for any $i=1 , \cdots, n$ and $H\subseteq \{x: D_{i}\rightarrow x, i= 1, \cdots, n\}$. $\forall (d_{1},\cdots, d_{n})\in \prod^{n}_{1}D_{i}$, $d_{i}\ll h_{i}$. Since $H \subseteq\bigcup_{1}^{n}\Ua d_{i}\in d(x)$, then there exists some $ F\in \mathcal{F}$ such that $\uparrow F\in\bigcup_{1}^{n}\Ua d_{i}\subseteq \bigcup^{n}_{1}\uparrow d_{i}$. Thus $\mathcal{F}\Rightarrow_{\mathcal{Q}}\uparrow F$ be Definition \ref{5.1}.
\end{proof}

\begin{prop}\label{5.5}
Let $X$ be a continuous QFS-space, then $\mathcal{Q}_{fin}(X)$ is a FS-space.
\end{prop}

\begin{proof}
Let $X$ be a QFS-space and $\mathcal{D}$ be a quasi-approximate identity for $X$ consisting of quasi-finitely separating
maps. Let $\varepsilon: \mathcal{Q}_{fin}(X)\rightarrow \mathcal{P}^{w}(\mathcal{Q}_{fin}(X))$ be $\varepsilon(\uparrow F)=\uparrow \bigcup_{x\in F}\delta(x)$. Now we show that
$\mathcal{D}^{*}=\{\varepsilon: \delta\in \mathcal{D}\}$ is a quasi-approximate identity for $\mathcal{Q}_{fin}(X)$ consisting of quasi-finitely separating maps.

(1) $\varepsilon$ is quasi-finitely separating map.

(i) For any $\uparrow F_{1}, \uparrow F_{2}\in Y$ and $\uparrow F_{2}\subseteq \uparrow F_{1}$. $\varepsilon(\uparrow F_{2})=\uparrow \bigcup_{x\in F_{2}}\delta(x)=\bigcup_{x\in \uparrow F_{2}}\uparrow \delta(x)\subseteq \bigcup_{x\in \uparrow F_{1}}\uparrow \delta(x)=\uparrow \bigcup_{x\in F_{2}}\delta(x)=\varepsilon(\uparrow F_{1})$ .

(ii) Since $\delta$ is quasi-finitely separating, there exists a finite set $F_{\delta}\subseteq X$ such that
for any $x\in X$, there exists $y\in F_{\delta}$ with $x\in \uparrow y\subseteq \uparrow \delta(x)$. Let $F_{\varepsilon}=\{\uparrow\bigcup_{x\in G}\delta(x): G\neq \emptyset,G\subseteq F_{\delta}\}$. For any $\uparrow F\in \mathcal{Q}_{fin}(X)$, let $D_{F}=\{y\in F_{\delta}: x\in \uparrow y\subseteq \uparrow \delta(x)\}$, then $D_{F}\in F_{\varepsilon}$. Since $F\subseteq \uparrow D_{F}$,
$\uparrow F\subseteq \uparrow\bigcup_{x\in D_{F}}\delta(x) \subseteq\uparrow \bigcup_{x\in F}\delta(x)$.

(iii) Let $\mathcal{F}=\{\uparrow F: F\in \mathcal{P}^{w}(X)\}\subseteq \mathcal{Q}_{fin}(X)$ be a directed set(respect to order  $\leq_{\mathcal{Q}}$) and $\uparrow H\subseteq\mathcal{Q}_{fin}(X)$.
If $\mathcal{F}\Rightarrow_{\mathcal{Q}}\uparrow H \Leftrightarrow $ there exists finite directed sets $D_{1}, \cdots, D_{n}\subseteq X$ such that

1. $H\cap \{x: D_{i}\rightarrow x\}\neq\emptyset$ for any $i=1 , \cdots, n$;

2. $H\subseteq \{x: D_{i}\rightarrow x, i= , \cdots, n\}$;

3. $\forall (d_{1},\cdots, d_{n})\in \prod^{n}_{1}D_{i}$, there exists some $\uparrow F\in \mathcal{F}$ such that $\uparrow F\in \bigcup^{n}_{1}\uparrow d_{i}$.

Let $\uparrow \bigcup_{x\in H}\delta(x)\subseteq U\in \tau$. For any $h_{i}\in H$, there exists directed subset $D_{i}$ of $X$ such that $D_{i}\rightarrow_{\tau} x$. Then $\delta(D_{i})\Rightarrow_{\tau} \delta(h_{i})$. Since $X$
 is a continuous space. Then  $\uparrow\delta(d_{i_{0}})\subseteq U$ for some $d_{i_{0}}\in D_{i}$. Then there exists some $\uparrow F\in \mathcal{F}$ such that $\uparrow F\in \bigcup^{n}_{1}\uparrow d_{i_{0}}$. It is easy to see that $\uparrow \bigcup_{x\in F}\delta(x)\subseteq \uparrow \bigcup_{x\in \{d_{i_{0}}: h_{i}\in H\}}\delta(x)\subseteq  U$. This is show that $\varepsilon (\mathcal{F})\Rightarrow _{\mathcal{Q}} \varepsilon(\uparrow H)$ by Proposition\ref{5.4}. That is $\varepsilon$ is a continuous functions.

By Definition \ref{3.1}, $\mathcal{D}^{*}$ is a set consisting of quasi-finitely separating maps on $Y$.

(2) Then we show that $\mathcal{D}^{*}$ is a quasi-approximate identity for $f(X)$. $\mathcal{D}^{*}$ is directed since $\mathcal{D}$ is directed. $\uparrow F \in \mathcal{Q}_{fin}(X)$, $\mathcal{D}^{*}(\uparrow F)=\{\varepsilon(\uparrow F):\varepsilon\in \mathcal{D}^{*}\}=\{\uparrow \bigcup_{x\in F}\delta(x): \delta\in \mathcal{D}\}$. Let $F\subseteq U\in \tau$. For any $f_{i}\in F$, $i=1,\cdots ,n$, $\mathcal{D}(f_{i})\rightarrow_{\tau} f_{i}$, then $\delta_{i}(f_{i})\subseteq U$. Since $\mathcal{D}$ is directed, there exist $\delta\in \mathcal{D}$ such that $\delta(f_{i})\subseteq \uparrow\delta_{i}(f_{i})$, $i=1,\cdots ,n$.
Since $\uparrow \delta(f_{i})\subseteq U$ and $\bigcup_{x\in F}\uparrow \delta(x)\subseteq U$, $\mathcal{D}^{*}(\uparrow F)\Rightarrow_{\mathcal{Q}}\uparrow F$ by Proposition\ref{5.4}. This shows that $\mathcal{D}^{*}$ is a quasi-approximate identity for $f(X)$. So, $\mathcal{Q}_{fin}(X)$ is a FS-space by Definition \ref{2.7}.

\end{proof}

\vspace{1cm} \noindent {\bf Acknowledgments} \small
\def\toto#1#2{\centerline{\hbox to0.7cm{#1\hss}
\parbox[t]{13cm}{#2}}\vspace{2pt}}

This work is supported by the NSFY of China (Nos. 11401435). The authors are grateful
to the referees for their valuable comments which led to the improvement of this paper.

\vspace{1cm} \noindent {\bf Conflict of interest} \small
\def\toto#1#2{\centerline{\hbox to0.7cm{#1\hss}
\parbox[t]{13cm}{#2}}\vspace{2pt}}

The authors declare that there is no conflict of interest in this paper.

\vspace{1cm} \noindent {\bf } \small
\def\toto#1#2{\centerline{\hbox to0.7cm{#1\hss}
\parbox[t]{13cm}{#2}}\vspace{2pt}}


\noindent{\bf References}

\begin{thebibliography}{99}
\newcommand{\DOI}[1]{doi: \href{https://doi.org/#1}{#1}}
\bibitem{1}X. Q. Xu, J. B. Yang, Topological representation of distributive hypercontinuous lattices, Chiness Annals of Mathematics, 2009, 30B: 199-206.
\bibitem{2}R. M. Amadio, P. L. Curien, Domains and Lambda-Calculi, Cambride:Cambride University Press, 1998.
\bibitem{3}J. B. Yang, M. K. Luo, Priestley spaces, quasi-hyperalgebraic lattices and Smyth powerdomains, Acta mathematics Sinica, 22(2006), 951-958.
\bibitem{4}J. Goubault-Larrecq, Non-Hausdorff topology and domain theory: Selected topics in point-set topology, Cambridge
University Press, (2013).
\bibitem{5}L. Zhou, Q. Li, Convergence On Quasi-continuous Domain, Journal of Computational Analysis and Applications, 2013,
15(2).

\bibitem{6}Q. Y. He, L. S. XU, X. Y. Xi, Consistent Smyth Powerdomains of Topological Spaces and Quasicontinuous Domains, Topology and its Applications,2017,228: 327-340. \DOI{10.1016/j.topol.2017.06.016}.
\bibitem{7}M. Ern$\acute{e}$, The ABC Order and Topology, Berlin: Hedermann , 1991.
\bibitem{8}G. Gierz, K. H. Hofmann, K. Keimel, J. D. Lawson, M. Mislove, D. S. Scott, Continuous Lattices and Domains, Cambridge
University Press, Cambridge, 2003.

\bibitem{9}Y. X. chen, H. Kou, Z. C.lyu, Continuity and core compactness of topological spaces, arXiv:2203.06344.
 \DOI{10.48550/arXiv.2203.06344}.
\bibitem{10}X. L. Xie, H. Kou, Lower power structures of directed spaces, J. Sichuan Univ., 57 (2020), 211¨C
217. \DOI{10.3969/j.issn.0490-6756.2020.002}.
\bibitem{11}S. Luo, X. Xu, On Monotone Determined Spaces, Electronic Notes in Theoretical Computer Science, 2017, 333: 63-72. \DOI{10.1016/j.entcs.2017.08.006}.
\bibitem{12} H. Miao, Q. Li, D. Zhao, On two problems about sobriety of topological spaces, Topology and its Applications, 2021, 295:
107667. \DOI{10.1016/j.topol.2021.107667}
\bibitem{13} X. L. Xie, H. Kou, The Cartesian closedness of c-spaces,AIMS Mathematics
2022, 7(9): 16315-16327. \DOI{10.3934/math.2022891}

\bibitem{14}G. Gierz, J. Lawson, Generalized continuous and hypercontinuous lattices, The Rocky Mountain Journal of Mathematics,
1981, 11(2): 271-296.
\bibitem{15}M. Ern$\acute{e}$, Scott convergence and Scott topology in partially ordered sets II, Continuous Lattices, Springer, Berlin, Heidelberg,
1981, 61-96.
\bibitem{16}Y. Yu, H. Kou, Directed spaces defined through $T_{0}$ spaces with specialization order (in Chinese), Journal of Sichuan University
(Natural Science Edition), 2015, 52(2): 217-222. \DOI{10.3969/j.issn.0490-6756.2015.02.001}.
\bibitem{17}M. Ern$\acute{e}$, Infinite distributive laws versus local connectedness and compactness properties,
Topology and its Applications, 156, 2054-2069 (2009). \DOI{10.1016/j.topol.2009.03.029}.
\bibitem{18}H. Kou, Directed spaces: An extended framework for domain theory, 1th Pan Pacific
International Conference on Topology and Applications Min Nan Normal University,
Zhangzhou City(2015), 11, 25-30.
\bibitem{19}M. Che, H. Kou, A cartesian closed full subcategory of c-spaces (in Chinese), Journal of Sichuan Normal University(Natural
Science), 2020, 43(06): 756-762. \DOI{10.3969/j.issn.1001-8395.2020.06.005}.

\bibitem{20}K. KEIMEL. Topological cone: functional analysis in $T_{0}$-setting, Semigroup Forum, 2008, 77(1):109-142. \DOI{10.1007/s00233-008-9087-0}.
\bibitem{21}H. Feng, H. Kou, Quasicontinuity and meet-continuity of T0 spaces(in Chinese), Journal of Sichuan University (Natural
Science Edition), 2017, 54(05): 905-910. \DOI{10.3969/j.issn.0490-6756.2017.05.002}.
\bibitem{22}X. Y. Zhang, Y. W. Hou, W. F. Zhang. The properties of continuous space(in Chinese), Fuzzy Systems and Mathematics, 2017, 31(2): 57-61.

\bibitem{23}W. Wu, H. L. Zhang, Y. Wang, The Way Below Bases of Directed Spaces(Iin chinese), Journal of Systems Science and Mathematical Sciences, 2022,42(04):1060-1066. \DOI{10.12341/jssms21243}.
\bibitem{24}R. Heckmann, K. Keimel, Quasicontinuous Domains and the Smyth Powerdomain, Electronic Notes in Theoretical Computer Science 298 (2013) 215¨C232. \DOI{10.1016/j.entcs.2013.09.015}.
\bibitem{25} X. L. Xie, Y. X. Chen, H. Kou, Power structures of directed spaces, arXiv:2203.04506. \DOI{10.48550/arXiv.2204.09926}.
\end{thebibliography}
\end{document}